\documentclass[12pt]{article}

\usepackage{amsmath,amsfonts,amssymb,amsthm}

\title{On the classification of positions and of complex structures in Banach spaces}

\author{Razvan Anisca\footnote{ R. Anisca was supported in part by NSERC Grant 312594-10}, Valentin Ferenczi\footnote{V. Ferenczi acknowledges the sup\-port of FAPESP pro\-ject 2013/11390-4 and project 2015/17216-1}\ \ and Yolanda Moreno\footnote{Y. Moreno has been supported in part by project MTM2010-20190-C02-01 and the program Junta de Extremadura GR15152 IV Plan Regional I+D+i, Ayudas a Grupos de Investigaci\'on.}}

\newcommand\address{\it {\noindent\leavevmode%
R. Anisca: Department of Mathematical Sciences, Lakehead University, Thunder Bay, ON, P7B 5E1, Canada \\
\\
V. Ferenczi: Departamento de Matem\'atica, Instituto de Matem\'atica e Estat\'istica 
Universidade de S\~ao Paulo, 05311-970 S\~ao Paulo, SP, Brazil \\ \\
Y. Moreno: Departamento de Matem\'aticas, Escuela Politecnica \\
Universidad de Extremadura, Avda. Universidad s/n, 10003 Cáceres, Spain}}

\date{}
\newtheorem{fact}{Fact}[section]
\newtheorem{thm}[fact]{Theorem}
\newtheorem{pr}[fact]{Proposition}

\newtheorem{cor}[fact]{Corollary}
\newtheorem{de}[fact]{Definition}
\newtheorem{lem}[fact]{Lemma}
\newtheorem{quest}[fact]{Question}

\newcommand{\R}{\mathbb{R}}
\newcommand{\C}{\mathbb{C}}
\newcommand{\N}{\mathbb{N}}

\newcommand{\Q}{\mathbb{Q}}

\begin{document}

\maketitle

\begin{abstract} A topological setting is defined to study the complexities of the relation of  equivalence of embeddings (or "position") of a Banach space into another and of the relation of isomorphism of complex structures on a real Banach space. The following results are obtained: a)  if $X$  is not uniformly finitely extensible, then there exists a space $Y$ for which the relation of position of $Y$ inside $X$ reduces the relation $E_0$ and therefore is not smooth; b) the relation of position of $\ell_p$ inside $\ell_p$, or inside $L_p$, $p \neq 2$, reduces the relation $E_1$ and therefore is not reducible to an orbit relation induced by the action of a Polish group;
c) the relation of position of a space inside another can attain the maximum complexity $E_{{\rm max}}$; d) there exists a subspace of $L_p, 1 \leq p <2$, on which isomorphism between complex structures reduces $E_1$ and therefore is not reducible to an orbit relation induced by the action of a Polish group.
\end{abstract}


\


\section{Introduction}
In this paper we are interested in defining a topological setting to compute the complexity of certain natural equivalence relations appearing in the theory of positions and/or complex structures. Our objective is to provide examples towards the idea that these relations are not "well classifiable", or in other words, to obtain high lower bounds of complexity for natural instance of these relations.  Our starting points are the previous results in which a continuum of equivalence classes were already obtained, without information on the complexity of the associated equivalence relation: examples of spaces with a continuum of mutually non isomorphic complex structures  \cite{A}, or examples of classical spaces with continuum many different positions inside another, see \cite{castplich} and \cite{MP}.

\

In this introduction we recall some basics of the theories of positions of Banach spaces, of complex structures, as well as of classification of analytic equivalence relations on Polish spaces.
In section 2, after defining the appropriate topological setting, we obtain lower bounds for the complexity of position of a space inside another, in different cases. We prove that
if $X$  is not uniformly finitely extensible, then there exists a space $Y$ for which the relation of position of $Y$ inside $X$ reduces the relation $E_0$ and therefore is not smooth (Theorem  \ref{NonUfo}).  Through a result about complexity of positions inside $\ell_p$-sums of non uniformly extensible spaces (Proposition \ref{E1}), we extend this and prove that the relation of position of $\ell_p$ inside $\ell_p$, or inside $L_p$, $p \neq 2$, reduces the relation $E_1$ and therefore is not reducible to an orbit relation induced by the action of a Polish group, Theorem \ref{E1lp}. Then through the study of complemented positions we use the main result of \cite{FLR} to show that the complexity of positions may be $E_{\rm max}$, the maximum complexity of analytic equivalence relations, Proposition \ref{Emax}. We end the section by providing the appropriate topological setting to study complex structures. In section 3, we describe an example to prove that there exists a subspace of $L_p, 1 \leq p <2$, on which isomorphism between complex structures reduces $E_1$ and therefore is not reducible to an orbit relation induced by the action of a Polish group.

\

\subsection{Positions of Banach spaces}
The notion of relative positions of Banach spaces arose in \cite{CM} where the definition of {\it automorphic} space was first introduced in connection with a classical result of Lindenstrauss and Rosenthal \cite{lindtzaf1}: {\it $c_0$ has the property that every isomorphism
between two of its infinite codimensional subspaces can be extended to an automorphism of the
whole space}. A separable space with such a property is said to be automorphic, or in other words, all its subspaces are in the same "position". The following problem remains open.

\begin{quest} Are $c_0$ and $\ell_2$ the only separable Banach spaces with that property?
\end{quest}

The papers \cite{CM, MP, castplich, CFM} were devoted to the study of different aspects of the automorphic problem. In \cite{MP, castplich} in particular, it is provided a general theory of positions for subspaces of a Banach space, by defining equivalent embeddings. Namely, given two infinite codimensional embeddings $T, U : Y \to X$ between separable Banach spaces, we let $\sim$ be the equivalence relation: $T \sim U$ if and only if there exists an automorphism $A$ of $X$ such that $T = AU$.  A {\em position} of $Y$ in $X$ is an $\sim$-equivalence class on the set of infinite codimensional embeddings from $Y$ to $X$.

\

The notion of automorphy index $a(Y,X)$ was introduced in \cite{MP} and it measures how many different positions a space $Y$ admits in another Banach space $X$. The automorphy index of $X$ is defined as $a(X) = \sup_Y a(Y, X)$ and, of course, a Banach space is said to be automorphic if $a(X) = 1$.  In \cite{castplich} it is estimated the automorphy indices $a(Y,X)$ for classical Banach spaces. The authors obtain, among other results: $a(c_0,X) = \{0, 1, 2,\aleph_0\}$ for every separable Banach
space X; $a(Y, \ell_p) = c$ for all subspaces of $\ell_p$, $p\not = 2$, and $a(Y,L_p) = c$ for all subspaces of $L_p$,
$p > 2$ not isomorphic to $\ell_2$; while $a(\ell_2, L_p) = 1$; for $1<p<2$ one has $a(Y,L_p) = c$ for all nonstrongly
embedded subspaces of $L_p$; $a(Y,L_1) = c$ for all nonreflexive subspaces of $L_1$, while
$a(\ell_2,L_1) = c$; $a(Y,C[0, 1])  \in \{1, c\}$ for every separable Banach space $Y$.

\

So once we have defined a topological setting for embeddings of a space $Y$ into another space $X$ and for the relation of being in the same position, we shall prove that the complexity of this relation is high for some of the above examples. Here, this can be interpreted as measuring the difficulty, once two embeddings $T,U$ of $Y$ into $X$ are given, of determining whether there exists an automorphism $A$ proving that these embeddings correspond to the same position.

\

We shall need   the notion of {\em uniformly finitely extensible} space considered
in \cite{castplich}. A space
 is uniformly finitely extensible (or UFO) if there exists $\lambda \geq 1$ such that for every finite dimensional subspace $E \subset X$, each linear operator $t: E \rightarrow X$ may be extended to a linear operator $T: X \rightarrow X$ with $\|T\| \leq \lambda \|t\|$. In \cite{CFM} it was proved that the UFO property is equivalent to being {\em compactly extensible}, meaning that every compact operator from a subspace $Y$ of $X$ into $X$ may be extended to the whole space.
Note that ${\mathcal L}_\infty$-spaces satisfy this property.

According to \cite{castplich} every automorphic space is UFO, and conversely, any UFO space is either an ${\mathcal L}_\infty$-space or a weak type $2$ near-Hilbert space  with the Maurey projection property. It remains open whether the UFO property is equivalent to being either ${\mathcal L}_\infty$-space or Hilbert.

\subsection{Complex structures}

A second theory that we shall revisit from the point of view of "definable" equivalence relations is the one of {\it complex structures} on real Banach spaces.
A real Banach space $(X, \| \cdot \|)$ admits a {\it complex structure} if there exists a multiplication of the elements of $X$ by complex scalars which is compatible with the norm:
$$
\|\lambda x\|=|\lambda| \|x\|, \qquad \forall x\in X, \qquad \forall \lambda\in \mathbb{C},
$$
or compatible with an equivalent norm to $\|\cdot\|$.

The complex structures on a real Banach space $(X, \| \cdot \|)$ correspond (in a one-to-one manner) to the $\mathbb{R}$-linear isomorphisms $T$ on $X$ satisfying $T^2=-Id$: if there is a complex structure we can take $Tx=ix$; conversely, we can define $(a+ib)x=ax+bTx$ which is compatible with the norm
$$
|\|x|\|=\sup_{\theta \in [0, 2\pi]} \|(\cos \theta)x+(\sin \theta)Tx\|.
$$

The isomorphic theory of complex structure addresses questions of existence, uniqueness, or the possible structure of the set of complex structures up to isomorphism.

By employing probabilistic methods, S. Szarek constructed in \cite{Sz} the first example of an infinite dimensional real Banach space which does not admit a complex structure. Using similar methods, J. Bourgain \cite{Bou} exhibited an example of an infinite dimensional complex Banach space $X$ not isomorphic to its {\it complex conjugate} $\overline X$:  $\overline X$ has the same elements and norm as $X$, the same addition of vectors, while the multiplication by scalars is
given by $\lambda \odot x=\overline{\lambda}x$, for $\lambda \in {\bf C}$, $x\in X$. Since it is clear that $X$ and $\overline X$ are identical as real Banach spaces, Bourgain's construction provides an example of a real Banach space with at least two non-isomorphic complex structures.

The  work of V. Ferenczi \cite{F} shows that it is possible to construct, for all positive integers $n\ge 1$, explicit examples of infinite dimensional real Banach spaces which admit precisely $n$ complex structures, up to isomorphim. W. Cuellar Carrera \cite{CC} gave an example of a separable real Banach space with exactly infinite countably many complex structures, up to isomorphism, while R. Anisca \cite{A} constructed subspaces of $L_p$, for $1\le p <2$, with a continuum of complex structures, up to isomorphism.

\subsection{Theory of complexity of equivalence relations}

We recall the theory of classification of analytic equivalence relations on Polish spaces by Borel reducibility. This area of research originated from the works of H. Friedman and L. Stanley \cite{FS} and independently from the works of L.A. Harrington, A.S. Kechris and A. Louveau \cite{HKL}. It may be thought of as an extension of the notion of cardinality in terms of complexity, when one counts equivalence classes.

A topological space is Polish if it is separable and its topology may be generated by a complete metric. Its Borel subsets are those belonging to the smallest $\sigma$-algebra containing the open sets. An analytic subset is the continuous image of a Polish space, or equivalently, of a Borel subset of a Polish space.
If $R$ (respectively $S$) is an equivalence relation on a Polish space $E$ (respectively $F$), then it is said that $(E,R)$ is {\em Borel reducible} to $(F,S)$, $(E,R) \leq_B (F,S)$, if there exists a Borel map $f:E \rightarrow F$ such that $\forall x,y \in E, x R y \Leftrightarrow f(x) F f(y)$.
They are {\em Borel bireducible}, $(E,R) \sim_B (F,S)$, if both $(E,R) \leq_B (F,S)$ and
$(F,S) \leq_B (E,R)$ hold. The aim is then to compare analytic equivalence relations modulo $\sim_B$.

One may note that such a map $f: E \rightarrow F$ induces an injection of $E/R$ into $F/S$ and therefore there are at least as many $S$-classes in $F$ as $R$-classes in $E$ when $(E,R) \leq_B (F,S)$. However the requirement that $f$ is Borel will induce much finer topological regularities;  actually the theory of $\leq_B$-classification is really interesting when both relations have $2^\omega$ classes, and there is a huge variety of such relations which are not bireducible to each other.

We now list a few important equivalence relations on the $\leq_B$-scale.
After the relations with finitely or countably many classes, the simplest Borel equivalence relation is $(\R,=)$, equality between real numbers. Actually by a result of Silver \cite{silver}, any Borel equivalence relation admits at most countably many classes, or there is a Borel reduction of $(\R,=)$ to it.
The analytic equivalence relations which are Borel reducible to $(\R,=)$ are called {\em smooth}; these are the relations which admit the real numbers as complete invariants.

 An important equivalence relation is the relation $E_0$ of eventual agreement between sequences of $0$ and $1$'s: on $2^{\omega}$,
$$\alpha E_0 \beta \Leftrightarrow  \exists m \in \N: \forall n \geq m, \alpha(n)=\beta(n).$$
The relation $E_0$ is a Borel equivalence relation with continuum many classes and which, furthermore, is non-smooth. So $(\R,=) <_B E_0$.  In fact $E_0$ is the $\leq_B$ minimum non-smooth Borel equivalence relation \cite{HKL}. Therefore, the most natural criterium to prove that an analytic relation is non-smooth is to reduce $E_0$ to it.

Quite natural are the orbit equivalence relations induced by the continuous  action of a Polish group $H$ on a Polish space $X$: the relation $E_H$ is defined on $X$ by
$$x E_H y \Leftrightarrow \exists h \in H: y=hx,$$
and is easily seen to be analytic. The relation $E_0$ is one of them. For any Polish group $H$, it is possible to prove that there is
a relation which is maximum among all orbit relations induced by actions of $H$. There is also a maximum $E_G$ for orbit relations associated to the action of  Polish groups; in particular
$E_0 <_B E_G$.

In 1997, Kechris and Louveau \cite{KL} discover that there are analytic equivalence relations which are not reducible to any orbit equivalence relation, or in other words, to $E_G$. There is actually a minimal equivalence relation, called $E_1$, among those which are not Borel reducible to an orbit equivalence relation. It is defined as the eventual agreement between sequences of real numbers: for $x,y \in \R^\N$,
$$x E_1 y \Leftrightarrow \exists n \forall m \geq n\ x_m=y_m.$$
The relation $E_1$ is, up to now, the only known obstruction to  reducibility to an orbit equivalence relation: $E_1 \not\leq_B E_G$.

Finally, the complete analytic equivalence relation $E_{\rm max}$ is the most complex
of all analytic equivalence relations, and is strictly above $E_1$ and $E_G$. It may be defined formally as the $\leq_B$-maximum equivalence relation, and the proof of its existence uses certain universality properties of
analytic sets.
There  also exist explicit realizations of $E_{\rm max}$, the most important in our setting being the relation of linear isomorphism between separable Banach spaces \cite{FLR}.

\section{Classification of subspaces, positions and complex structures}

In what follows, the notation $\simeq$ will be used for equivalence relations associated to linear isomorphism of Banach spaces, while
$\sim$ will be used for relations of equivalence of positions of a space inside another.
The letters $T,U$ will usually stand for embeddings, $A,B$ for automorphisms, $P,Q$ for projections, $J,K$ for complex structures.

\

Since the definition of the order $\leq_B$ relies on the Borel property of the function realizing the reduction from a relation to another, the topologies of the Polish spaces considered only play a role through the Borel sets they generate. In the following we shall therefore prefer to talk about standard Borel spaces than Polish spaces: a {\em standard Borel space} is a set equipped with a $\sigma$-algebra which is the $\sigma$-algebra of Borel sets induced by some Polish topology on the set.

Let $X$ be a separable infinite dimensional Banach space.
There is a natural way to equip the set of infinite dimensional subspaces of  X with a Borel structure (see, e.g., \cite{Kechris}), and the relation we are more interested in, of linear isomorphism, is analytic in this setting \cite{Bo}. If we choose $X$ to be a universal space such as $C(2^\omega)$, then we obtain a description of the {\em standard Borel space $SB$} of all separable Banach spaces.
 What is proved in \cite{FLR} is that linear isomorphism $\simeq$ between elements of $SB$  is an analytic relation which is Borel bireducible to $E_{\rm max}$, or in other terms, has the maximum complexity among all analytic equivalence relations on Polish spaces. We may also restrict the relation $\simeq$ to $SB(X)$, the standard Borel space of infinite dimensional subspaces of $X$.

In \cite{FLR} some other relations are proved to have maximum complexity
$E_{\rm max}$: Lipschitz isomorphism or (complemented) biembeddability on $SB$, uniform homeomorphism of complete separable metric spaces, for example. According to \cite{LR}, isometric biembeddability also has complexity $E_{\rm max}$. On the other hand, linear isometry on $SB$ \cite{melleray} and homeomorphism of compact metric spaces \cite{Zi}, for example, are of complexity $E_G$.

\

Given $Y,X$ separable Banach spaces, we shall also need  to study analytic equivalence relations on ${\cal B}(Y,X)$, the set of bounded linear operators from $Y$ into $X$, or on some of its subsets. To do this we note that endowed with the strong operator topology, the space ${\cal B}(Y,X)_{\le 1}$ of linear operators with norm less than or equal to 1 is Polish, while ${\cal B}(Y,X)$ is a standard Borel space with respect to the Borel structure generated by the strong operator topology (as a countable union of standard Borel spaces).
This result may be found in \cite{Kechris} p80.

It will be useful to note that since multiplication of operators is continuous in the strong operator topology when restricted to $B \times {\cal B}(X) \longrightarrow {\cal B}(X)$, where $B$ is a norm bounded subset of ${\cal B}(X)$, the multiplication of operators ${\cal B}(X) \times {\cal B}(X) \longrightarrow {\cal B}(X)$ is Borel.

\subsection{Complemented subspaces}
In this part we aim to define a Borel standard space of {\em complemented} infinite dimensional subspaces of a given separable Banach space $X$. This is done as follows: since the multiplication of operators ${\cal B}(X) \times {\cal B}(X) \longrightarrow {\cal B}(X)$ is Borel, the set of projections on $X$ is a Borel subset of ${\cal B}(X)$.  Combined with the easy fact that the set of compact operators on $X$ is Borel as well, we deduce that the set of projections of infinite range is Borel in ${\cal B}(X)$ for the SOT topology.

\begin{de} Let $X$ be a separable infinite dimensional Banach space. We denote by
${\cal P}(X)$ the Borel standard space of projections of infinite range in ${\cal B}(X)$, equipped with the SOT topology.
\end{de}

\begin{de} Let $X$ be a separable infinite dimensional Banach space. The relation $\simeq$ defined on ${\cal P}(X)$ by
$$P \simeq Q \Leftrightarrow PX \simeq QX$$
is called the {\em relation of linear isomorphism between complemented subspaces of $X$}.
\end{de}

We may relate $({\mathcal P}(X),\simeq)$ to $(SB(X),\simeq)$ as follows:

\begin{pr} Let $X$ be a separable infinite dimensional Banach space. Then the map
from ${\cal P}(X)$ into $SB(X)$ defined by $P \mapsto PX$ is Borel.
In particular the relation $\simeq$ is analytic on ${\cal P}(X)$ and
$$({\cal P}(X),\simeq) \leq_B (SB(X),\simeq).$$
\end{pr}

\begin{proof} Let ${\mathcal O}_U$ be a typical Borel set generating the Effros Borel structure of $SB(X)$, i.e.
${\mathcal O}_U=\{Y \subset X: Y \cap U \neq \emptyset\}$, where $U \neq \emptyset$ is open.
Then given $(x_n)_n$ a dense family in $X$, we note that
$$PX \in {\mathcal O}_U \Leftrightarrow PX \cap U \neq \emptyset
\Leftrightarrow \exists n \in \N Px_n \in U.$$
This last condition is Borel in ${\mathcal P}(X)$.
 \end{proof}

This means that the complexity of isomorphism between complemented subspaces of $X$ will be $\leq_B$ below the the complexity of isomorphism between subspaces of $X$, via the set $\{PX, P \in {\mathcal P}(X)\}$. The result of \cite{FLR} about maximum complexity of isomorphism between subspaces may be extended to complemented subspaces as follows:

\begin{pr}\label{u} The complexity of linear isomorphism between complemented subspaces of $U$ is $E_{\rm max}$, where $U$ is Pe\l czy\'nski universal unconditional space.
\end{pr}

\begin{proof} It is proved in \cite{FLR} that $E_{\rm max}$ is Borel reducible to isomorphism between
subspaces generated by subsequences of the unconditional basis of a certain space, which therefore we may assume to be $U$ \cite{P}. Noting that every subsequence of the basis $(u_n)_{n \in A}$ of $U$ is complemented by the natural projection $P_A$, it is enough to prove that the map
$${\N}^{\N} \mapsto {\cal P}(U)$$ taking an infinite set $A$ to $P_A$ is Borel.
This is clear since
$$P_A \in {\mathcal O}_{P,x_1,\ldots,x_n,\epsilon} \Leftrightarrow
\forall i=1,\ldots,n, \; \|Px_i-P_A x_i\|<\epsilon$$
which is an open condition in ${\N}^{\N}$.
\end{proof}

\subsection{Complexity of positions}

Given  infinite dimensional separable Banach spaces $Y, X$, we shall use the notation ${\rm Emb}(Y,X)$ for the set of linear operators which are infinite codimensional embeddings of $Y$ into $X$ (i.e. onto infinite codimensional subspaces of $X$). We also denote by $GL(X)$ the group of automorphisms on $X$. We let $\sim$ be the equivalence relation on
${\rm Emb}(Y,X)$ defined by
$$T \sim U \Leftrightarrow \exists A \in GL(X): T=AU.$$

\begin{de} Let $X, Y$ be infinite dimensional and separable. A {\em position} of $Y$ in $X$ is an $\sim$-equivalence class on ${\rm Emb}(Y,X)$.
\end{de}

By the {\em complexity} of the positions of $Y$ in $X$, we mean the complexity of the equivalence relation $\sim$ on ${\rm Emb}(Y,X)$ along the $\leq_B$-scale. For this to make sense, we just need to note the following:

\begin{pr} The space ${\rm Emb}(Y,X)$ is a Borel standard space and $\sim$ is an analytic relation on it.
\end{pr}

\begin{proof}
It is an easy exercise to check that the space ${\rm Emb}(Y,X)$ is a Borel subset of ${\mathcal B}(Y,X)$ (recalling that these sets are equipped with the SOT), and therefore a Borel standard space in its own right. Fix $(y_n)_n$ and $(x_n)_n$  dense in $Y$ and $X$ respectively. 
We let $B \subset X^{\omega} \times {\rm Emb}(Y,X)^2$ be defined by
$$((z_n)_n,T,U) \in B \Leftrightarrow$$ $${\rm the\ map\ } x_n \mapsto z_n {\rm\ extends\ to\ an\ isomorphism\ } A{\rm \ on\ } X {\rm\ satisfying\ } T=AU.$$
We claim  that $B$ is Borel, and therefore $\sim$ is analytic by
$$T \sim U \Leftrightarrow \exists (z_n)_n \in X^\omega\  ((z_n)_n,T,U) \in B.$$
That $B$ is Borel follows from 
$$((z_n)_n,T,U) \in B \Leftrightarrow
\forall n \in \N, \forall \epsilon \in \Q^{+*}, \exists m \in \N\ 
\|x_n - z_m\| \leq \epsilon$$
$$\wedge\ \exists K \in \N \big(\forall (\lambda_i)_i \in c_{00}(\Q),
K^{-1} \|\sum_i \lambda_i x_i\| \leq \|\sum_i \lambda_i z_i\|
\leq K \|\sum_i \lambda_i x_i\| $$
$$\wedge\ \forall m,n \in \N, \forall \epsilon \in \Q^{+*},
 \|Uy_m-x_n\| \leq \epsilon \Rightarrow \|Ty_m-z_n\| \leq K\epsilon\big).$$
\end{proof}

We now turn to the notion of uniformly finitely extensible (or UFO) space recalled in the introduction.  Since every automorphic space has this property,  non-UFO spaces admit subspaces in at least two positions. We shall now extend this to prove that the relation of position is not even smooth in these instances. Recall that UFO spaces are either ${\mathcal L}_\infty$-spaces or near Hilbert, meaning that non-UFO spaces include most of the classical spaces.

\

Given $n \in \N$, we write for $\alpha,\beta \in 2^\omega$, $\alpha E_0^n \beta$ to mean that
$\alpha_i=\beta_i$ for all $i \geq n$. We also define for two embeddings
of $Y$ into $X$, $U \sim_n V$ to mean that there exists an automorphism
$T$ of $X$ with $TU=V$, such that ${\rm max}\{\|T\|, \|T^{-1}\|\} \leq n$.

\begin{thm}\label{NonUfo}
If $X$ is a separable, infinite dimensional, non uniformly finitely extensible space, then there is some subspace $Y$ of $X$
such that the relation $E_0$ is Borel reducible to $({\rm Emb}(Y,X),\sim)$. In particular the positions of $Y$ in $X$ are not smooth. \end{thm}

\begin{proof}
Since $X$ is not UFO there exists (see \cite{MP}) a subspace $Y\subset X$ admitting a finite dimensional decomposition $Y = \sum Y_n$ and a sequence of norm-one operators $T_n: Y_n\to X$ such that every extension of $T_n$ to $X$ has norm not less than $2^{2n}$. Let $\alpha\in 2^\omega$, we define an operator $T_\alpha: Y \to X$, $T_\alpha(\sum y_n) = \sum_n2^{-n}T_n^{\alpha(n)}y_n$, where
\[ T^{\alpha(n)}(y_n) =
\begin{cases}
T_ny_n & if \; \alpha(n) = 1\\
y_n & if \; \alpha(n) = 0
\end{cases}\]

The operator $T_\alpha$ is obviously compact, $\|T_\alpha\| \leq 1$ and does not admit any extension to an operator $X \to X$. Take now $0 < \varepsilon < 1$ and consider the $(1+\varepsilon)$-isometry $A_\alpha : Y \to X$, $A_\alpha = Id + \varepsilon T_\alpha$; this map does not admit any extension to $X$ as neither $T_\alpha$ does it. Let $Y_\alpha = A_\alpha(Y)$, we define the following map$$\left(2^\omega, E_0\right) \stackrel{f}\longrightarrow \left({\rm Emb}(Y, X), \simeq\right),\;\; f(\alpha) = A_\alpha.$$
The map $f$ is well defined, uniformly bounded, and Borel (actually continuous).

 Let us see that $\alpha E_0 \beta$ if and only if $Y_\alpha$ and $Y_\beta$ are in the same po\-si\-tion. Assume first that $\alpha E_0 \beta$ and let $m$ be such that $\alpha  E_0^m \beta$, then it follows that ${T_\alpha}_{|_{\sum_{n\geq m} Y_n}} = {T_\beta}_{|_{\sum_{n\geq m} Y_n}}$, and we can write $$Y_\alpha = A_\alpha (\sum_{n< m} Y_n) \oplus B\ {\rm and\ } Y_\beta = A_\beta (\sum_{n< m} Y_n) \oplus B,$$ where $ B = A_\alpha (\sum_{n\geq m} Y_n) = A_\beta (\sum_{n\geq m} Y_n)$. The (canonical) projections $P_\alpha : Y_\alpha \to Y_\alpha$ and $P_\beta : Y_\beta \to Y_\beta$,  with ranges $P_\alpha(Y_\alpha) = A_\alpha (\sum_{n< m} Y_n)$ and $P_\beta(Y_\beta) = A_\beta (\sum_{n< m} Y_n)$ of finite dimension, admit extensions  $\widehat{P_\alpha} : X \to A_\alpha(\sum_{n< m} Y_n)$ and $\widehat{P_\beta} : X \to A_\beta(\sum_{n< m} Y_n)$, respectively, which are also projections. Let us write $X_\alpha = (1_X - \widehat{P_\alpha})(X)$ and $X_\beta = (1_X - \widehat{P_\beta})(X)$, hence $B\subset X_\alpha\cap X_\beta$, and since $X = A_\alpha(\sum_{n< m} Y_n)\oplus X_\alpha = A_\beta(\sum_{n< m} Y_n)\oplus X_\beta$ we can easily define an automorphism $\tau$ of $X$ such that $\tau A_\alpha = A_\beta$ since the finite dimensional pieces have the same dimension and so $X_\alpha$ and $X_\beta$ are isomorphic (as all hyperplanes in a Banach spaces are). Let us note for future use that, by the well-known fact that all subspaces of codimension $k$ in a Banach space are $c(k)$-isomorphic, for some $c(k)$, we may deduce that $X_\alpha$ and $X_\beta$ are $c_m$-isomorphic, where $c_m$ only depends on $m$. So we can actually control the norms of
$\tau$ and $\tau^{-1}$ by some constant $c^{\prime}_m$ depending only on $m$, i.e $Y_\alpha \sim_{c^{\prime}_m} Y_\beta$ once $\alpha E_0^m \beta$. \\

On the other direction, we shall prove that if $\alpha$ and $\beta$ are not $E_0^m$-related and $n \geq m$ is such that   $\alpha_n = 0$ and $\beta_n = 1$, then any map $\tau: X \to X$ such that $\tau A_\alpha=A_\beta$ has norm at least  $\frac{1}{2}(\varepsilon 2^{m}-1)$.
This implies that if $\alpha$ and $\beta$ are not $E_0$-related (and without loss of generality, $\alpha_n=0$ and $\beta_n=1$ for infinitely many $n$) then there is no automorphism  $\tau : X \to X$ such that $\tau A_\alpha = A_\beta$.

So let $n \geq m$ be such that   $\alpha_n = 0$ and $\beta_n = 1$, and $\tau: X \to X$ be such that  $\tau A_\alpha=A_\beta$.
 In particular, $\tau_{|_{Y_\alpha}} = A_\beta A^{-1}_\alpha$, which means that $\tau : X\to X$ extends $A_\alpha^\beta = A_\beta A^{-1}_\alpha: Y_\alpha \to Y_\beta$. Since$$\tau(y) = A_\beta(y) - \varepsilon \tau T_\alpha (y),$$take $y=y_n\in Y_n$.

\begin{eqnarray*} \tau(y_n) &= &A_\beta(y_n) - \varepsilon \tau T_\alpha (y_n) \\
&=& y_n + \varepsilon 2^{-n}T_n(y_n) -\varepsilon 2^{-n}\tau(y_n)\end{eqnarray*}
whence
$$ (1 +\varepsilon 2^{-n})\tau(y_n) - y_n = \varepsilon  2^{-n} T_n(y_n)$$
then
$$\frac{2^n}{\varepsilon}\left(  ( 1+\varepsilon 2^{-n} )\tau(y_n) - y_n \right)= T_n(y_n),$$
so
$$\frac{2^n}{\varepsilon}\left(  ( 1+\varepsilon 2^{-n} )\tau - id \right)$$extends $T_n$. Therefore
$$2^{2n} \leq \| \frac{2^n}{\varepsilon}\left(  (  1+ \varepsilon 2^{-n})\tau - id \right) \| \leq \frac{2^{n} } {\varepsilon} (2\|\tau\|+1 )$$ and $\|\tau\| \geq \frac{1}{2}(\varepsilon 2^{n}-1) \geq \frac{1}{2}(\varepsilon 2^{m}-1)$.
\end{proof}

Beyond the automorphic space problem, it would be interesting to investigate for which $X$ the relation of position of $Y$ in $X$ is smooth, for all choices of $Y$.  The above shows that $X$ would have to be uniformly finitely extensible, leading to the question:

\begin{quest} Find a non automorphic, uniformly finitely extensible space $X$, such that the position of $Y$ inside $X$ is smooth for all subspaces $Y$ of $X$. \end{quest}

 We recall that it is an open conjecture (see \cite{CFM} and \cite{castplich}) whether the UFO property is equivalent to being either ${\mathcal L}_\infty$ or isomorphic to the Hilbert space.
 Among ${\mathcal L}_\infty$ (and therefore UFO) spaces which are not automorphic, it remains fairly open when $E_0$ can be reduced to positions of subspaces. For example,
remembering that in \cite{castplich} it is shown that $a(Y,C(0,1))=c$ for many choices of $Y$ (such as $Y=\ell_p$, $p \neq 1$, or $Y=C(0,1)$ itself):

\begin{quest} Find a space $Y$ such that $E_0$ is reducible to the positions of $Y$ inside $C(0,1)$. \end{quest}

\

Note that  the relation of equivalence of positions of $Y$ inside $X$  is the orbit relation of the action $(A,T) \mapsto AT$ of the group $GL(X)$ on the standard Borel space ${\rm Emb}(Y,X)$. So although $GL(X)$ is not a Polish group, one may be led to look for uniformity arguments
to prove that equivalence of positions of space $Y$ inside $X$ is reducible to an orbit relation induced by action of some Polish group, and in particular is not maximum among analytic equivalence relations.
One of our main results, however, is that this is not so.
To prove this we turn to reductions of the relation $E_1$ in the case of the classical spaces $\ell_p$ and $L_p$.

\begin{pr}\label{E1} Let $1\leq p <\infty$. Let $Y, X$ be separable.
Assume  there is a Borel reduction $r$ of $(2^\omega,E_0)$ to $({\rm Emb}(Y,X),\sim)$ with the following properties:
\begin{itemize}
 \item[(a)] $r(\alpha)$ is bounded uniformly on $\alpha \in 2^\omega$
 \item[(b)] there exists a sequence $(c_k)_k$ of integers such that
$$\alpha E_0^k \beta \Rightarrow r(\alpha) \sim_{c_k}  r(\beta)$$
\item[(c)] there exists a sequence $(d_k)_k$ of integers tending to infinity such that
if $\alpha$ and $\beta$ are not $E_0^k$-related  and we assume $\alpha(i)=0$ and $\beta(i)=1$ for some $i \geq k$, then there is no map $T$ on $X$ of norm less than $d_k$ such that $T r(\alpha) = r(\beta)$.
\end{itemize}
Then the relation $E_1$ is Borel reducible to $({\rm Emb}(\ell_p(Y), \ell_p(X)),\sim)$ and to\break $({\rm Emb}(c_0(Y), c_0(X)),\sim)$.
\end{pr}

\begin{proof} To each $\alpha=(\alpha_n)_n \in (2^\omega)^\omega$, where
for each $n \in \omega$, $\alpha_n=(\alpha_n(k))_{k \in \omega}$,
 associate
$$R(\alpha)=(r((\alpha_n(1))_{n \in \omega}), r((\alpha_n(2))_{n \in \omega}), \ldots, r((\alpha_n(k))_{n \in \omega}), \ldots).$$Because $r$ is bounded by (a), this defines an embedding of $\ell_p(Y)$ into $\ell_p(X)$ (of $c_0(Y)$ into $c_0(X)$), and $R$ is Borel. We denote by $Y_k$ and $X_k$ the $k$-th copies of $Y$ and $X$ respectively.

Note that if $\alpha E_1 \beta$ then there is $m\in\N$ such that for $k \geq m$, $\alpha_k=\beta_k$, which implies
that for each $k$, $\alpha_n(k)_n$ and $\beta_n(k)_n$ are $E_0$-related and actually only differ by at most the first $m$-terms.  Then by the property (b),  we may paste the maps $T_k: X_k \rightarrow X_k$ for which $T_k r(\alpha_n(k)_n)=r(\beta_n(k)_n)$ and ${\rm max}\{\|T_k\|, \|T_k^{-1}\|\} \leq c_m$, to define a global automorphism $T$ witnessing that $R(\alpha)$ and $R(\beta)$ are in the same position.

On the other hand assume $\alpha$ and $\beta$ are not $E_1$-related but that there is an automorphism $T$ of $\ell_p(X)$ such that $R(\alpha)=TR(\beta)$; let $K$ be an infinite subset of $\N$ such that for all $k \in K$, $\alpha_k \neq \beta_k$ and for each $k \in K$ let $i_k$ be such that $\alpha_k(i_k) \neq \beta_k(i_k)$.  Without loss of generality assume $\alpha_k(i_k)=0$ and $\beta_k(i_k)=1$ with $K$ still infinite.
Since $r(\alpha_n(i_k)_n)$ (resp. $r(\beta_n(i_k)_n)$) is an embedding of $Y_{i_k}$ into $X_{i_k}$, we have
$$r(\alpha_n(i_k)_n)= P_k T_k r(\beta_n(i_k)_n),$$
where $T_k: X_{i_k} \rightarrow \ell_p(X)$ (or $T_k: X_{i_k} \rightarrow c_0(X)$) is the restriction of $T$ to $X_{i_k}$ and $P_k$ the canonical projection onto $X_{i_k}$.
Note that $(\alpha_n(i_k))_n$ and $(\beta_n(i_k))_n$ are not
$E_0^k$-related, since $\alpha_k(i_k)=0 \neq 1=\beta_k(i_k)$. Therefore by (c) the map $P_k T_k$ has norm at least $d_k$. And therefore
$\|T\| \geq d_k$ for all $k \in K$, which is a contradiction.
\end{proof}

We finally prove that $E_1$ is reducible to positions of $Y$ inside $X$ for classical spaces such as the $\ell_p$'s, and therefore this relation is not reducible to the orbit  relation induced by the action of a Polish group.

\begin{cor}
Let $X$ be a separable non UFO space. There exists a subspace $Y\subset X$ such that $E_1$ is Borel reducible to $({\rm Emb}(\ell_p(Y), \ell_p(X)),\sim)$ and to $({\rm Emb}(c_0(Y), c_0(X)),\sim)$.
\end{cor}

\begin{proof}
Simply notice that the Borel map $f: (2^\omega, E_0) \to ({\rm Emb}(Y, X),\sim)$ in Proposition \ref{NonUfo} verifies the conditions in Proposition \ref{E1}.
\end{proof}

\begin{thm}\label{E1lp}  For $1\leq p < \infty$, $p\not=2$,
the relation $E_1$ is Borel reducible to  $({\rm Emb}(\ell_p, \ell_p),\sim)$
and to $({\rm Emb}(\ell_p,L_p),\sim)$. Therefore the relation of position of $\ell_p$ in $\ell_p$ (resp. in $L_p$) is not Borel reducible to an orbit relation induced by the action of a Polish group.\end{thm}

\begin{proof}
 By Proposition 3.15 in \cite{castplich}, there exists a subspace $Y\subset \ell_p$, admitting a FDD, which is isomorphic to $\ell_p$
 and for which there is a Borel reduction $(2^\omega, E_0)\to ({\rm Emb}(Y, \ell_p),\sim)$ which also verifies conditions in Proposition \ref{E1}.
The same arguments based on \cite{castplich} Proposition 3.15 work  for $X = L_p$.
\end{proof}

\subsection{Complexity of complemented positions}

A classical method to compute equivalent positions is to look at embeddings as complemented subspaces and compare the summands. Complemented positions will also be more easily related to the relation of  isomorphism of complex structures.

This motivates to define, for
 $X, Y$ be infinite dimensional separable Banach spaces,
 the {\em standard Borel space of complemented embeddings of $Y$ into $X$} as the Borel subspace
${\rm Emb}_{c}(Y,X)$
of ${\rm Emb}(Y,X) \times {\cal P}(X)$ given by
$$(T,P) \in {\rm Emb}_{c}(Y,X) \Leftrightarrow TY=PX.$$
We let $\sim$ be the analytic equivalence relation defined on ${\rm Emb}_{c}(Y,X)$ by
$$(T,P) \sim (U,Q) \Leftrightarrow T \sim U,$$ and we call the complexity of this relation the
{\em complexity} of complemented positions of $Y$ in $X$.

\begin{de} Let  $X, Y$ be infinite dimensional and separable. A {\em complemented position} of $Y$ in $X$ is an $\sim$-equivalence class on ${\rm Emb}_c(Y,X)$.\end{de}

We note that, as is to be expected:

\begin{pr}
The map ${\rm Emb}_{c}(Y,X) \rightarrow
{\rm Emb}(Y,X)$ defined by $(T,P) \mapsto T$ is Borel. In particular,
the complexity of complemented positions of $Y$ in $X$ is a lower bound of the complexity of positions of $Y$ in $X$.
\end{pr}

We shall now use Proposition \ref{u} to show that the highest complexity $E_{\rm max}$ among analytic equivalence relations, can be achieved for the complexity of (complemented) positions of a space inside another. So there will be no upper bound other than $E_{\rm max}$ for the complexity of positions of a space inside another.

\begin{pr}\label{Emax}
If $U$ is Pe\l czy\'nski universal unconditional basis, then the complexity of the relation of (complemented) positions of $U$ in itself is $E_{\rm max}$.
\end{pr}

\begin{proof}
By Proposition \ref{u} there is a Borel reduction $r$ of $E_{\rm max}$ to isomorphism between  subspaces of $U$ generated by subsequences of the basis $(u_n)_n$ (identified with elements $A$ of $\N^\N$).  Since $U \simeq \ell_2(U) \simeq U \oplus \ell_2(U)$ we may use as basis of $U$ a basis $(v_n)_n$ which is the union of infinitely countably many copies $(u_n^i)_n$ of $(u_n)$, $i=0,1,\ldots$. We denote $U_i=[u_n^i]_{n \in \N}$, $V_0=\oplus_{i \geq 1}U_i$ and see
$U$ as $U=\oplus_{n \in \N} U_n$.
Note that if $A \in \N^\N$, then $Y_A=[u_n^0]_{n \notin A} \oplus V_0$ is a complemented subspace of $U$ which is isomorphic to $U$ by classical properties of Pelczynski's space; therefore embeddings onto
subspaces $Y_A$ and $Y_B$ are in the same position if and only if the quotients $ U/Y_A$ and $ U/Y_B$ are isomorphic, i.e., $[u_n]_{n \in A} \simeq [u_n]_{n \in B}$. From this we deduce that there is a Borel reduction of $E_{\rm max}$ to the relation $\sim'$ on $\N^\N$ defined by
$$A \sim' B \Leftrightarrow {\rm embeddings\ onto}\  Y_A {\rm \ and\ } Y_B {\rm \ have\ the\ same\ position\ in\ } U.$$

Let $Q_A$ be the canonical projection onto $Y_A$ associated to the unconditional basis $(v_n)_n$ and $T_A$ a choice of an embedding of $U$ into $U$
for which $T_A(U)=Y_A$. Since
$$A \sim' B \Leftrightarrow T_A \sim T_B \Leftrightarrow (T_A, Q_A) \sim (T_B, Q_B),$$
it only remains to check that the map from $\N^\N$ to ${\rm Emb}_c (U,U)$ associating to $A$ the pair $(T_A,Q_A)$ may be chosen to be Borel.
 Since $A \mapsto Q_A$ is clearly Borel, let us describe a Borel choice of $A \mapsto T_A$: we extend by linearity the map for which
$$T_A(u_n^i)=u_n^{i+1}, \forall n \in A, \forall i \in \N$$
and
$$T_A(u_n^i)=u_n^i, \forall n \notin A, \forall i \in \N.$$
This is a Borel map for which $T_A$ is an embedding of $U$ onto $Y_A$, for all $A$.
\end{proof}

\subsection{Complex structures}

The set of complex structures on a separable real space $X$ will be identified with the set
$${\cal C}(X):=
\{T\in {\cal B}(X) \ |\ T^2=-Id\}.
$$
Since the multiplication of operators ${\cal B}(X) \times {\cal B}(X) \longrightarrow {\cal B}(X)$ is Borel,
it follows that the set ${\cal C}(X)$ is  a Borel subset of ${\cal B}(X)$ and therefore a standard Borel set.

\begin{de} Let $X$ be a separable real Banach space. The set
$${\cal C}(X)=\{T\in {\cal B}(X) \ |\ T^2=-Id\}$$ seen as a subspace of ${\cal B}(X)$ with the strong operator topology, will be called the {\em standard Borel space of complex structures on $X$}.
\end{de}

\begin{de} Given two elements $J, K$ in ${\cal C}(X)$, we say that $J\simeq K$ if there exists an isomorphism $A\in GL(X)$ such that
$AJA^{-1}=K$. \end{de}

Note that $J \simeq K$ if and only if the associated complex structures are $\C$-linearly isomorphic.

\begin{lem}
The relation $\simeq$ is analytic on the standard Borel space ${\cal C}(X)$.
\end{lem}

\begin{proof} We fix a countable family $(x_n)_n$ with dense linear span in $X$, and note that
an isomorphism $A$ on $X$ may be coded by a family $(y_n)_n \in X^{\omega}$ with dense linear span
and so that the map $x_n \mapsto y_n$ extends to an isomorphism onto its image.
The relation $AJA^{-1}=K$ is then equivalent to
$AJ  x_n= K y_n$ for all $n$, which we may reformulate, using an upper bound $M$ for $\|A\|$ and $\|K\|$, in terms of approximations of
$K(\sum_i \lambda_i x_i)$ when $\sum_i \lambda_i x_i$ approximates $y_n$.
We deduce the following characterization: $J \simeq K$ if and only if there exists
$(y_n)_n \in X^{\omega}$ such that
\begin{itemize}
\item[(i)] $\exists k \in \N, \forall (\lambda_n) \in c_{00}(\Q),
k^{-1}\|\sum_n \lambda_n x_n\| \leq \|\sum_n \lambda_n y_n\| \leq k \|\sum_n \lambda_n x_n\|,$
\item[(ii)] $\forall n \in \N, \forall q \in \Q^{+*}, \exists (\lambda_i) \in c_{00}(\Q):\
\| x_n - \sum_i \lambda_i y_i\| <q,$
\item[(iii)] $\exists M \in \N\ \forall q \in \Q^{+*}, \forall  (\lambda_i), (\mu_i) \in c_{00}(\Q) ,
 \|K(\sum_i \lambda_i x_i)\| \leq M \|\sum_i \lambda_i x_i\|$

{\rm and\ } $\forall n \in \N,
\big( ( \| \sum_i \lambda_i x_i - J x_n\| \leq q) \wedge ( \|y_n-\sum_i \mu_i x_i\| \leq q)\big)$

$ \Rightarrow \| \sum_i \lambda_i y_i -K(\sum_i \mu_i x_i) \| \leq 2M q $.
\end{itemize}
Since the set of $((y_n), J, K)$ satisfying (i)-(ii)-(iii) is a Borel subset of the space $(X,\|.\|)^{\omega} \times {\cal C}(X)^2$, it follows that $\simeq$ is analytic on ${\cal C}(X)$.
\end{proof}

We now relate the Borel standard space ${\mathcal C}(X)$ of complex structures to a Borel standard space of complemented subspaces as follows. Recall that if $J$ is a complex structure on a real space $X$, then the space
$$X_J=\{(x,Jx), x \in X\}$$ is a complemented, $\C$-linear subspace of the complexification $X \oplus_{\C} X$ of $X$, which is $\C$-linearly isomorphic to the complex structure $X^J$ induced by $J$. Actually, $X_J$ is the image of the $\C$-linear projection $P_J$ defined on $X \oplus_{\C} X$ by
$$P_J(x,y)=\frac{1}{2}(x-Jy,Jx+y).$$
It is clear that the map $J \mapsto P_J$ is a Borel isomorphism between ${\mathcal C}(X)$ and
the Borel subspace $\{P_J, J \in {\mathcal C}(X)\}$ of ${\mathcal P}(X \oplus_{\C} X)$, and by definition,
$$J \simeq K \Leftrightarrow X^J \simeq X^K \Leftrightarrow P_J \simeq P_K.$$
In other words our definition of complexity of isomorphism bewteen complex structures coincide with the natural one induced by isomorphism of complemented subspaces of $X \oplus_\C X$ on the Borel set $\{P_J, J \in {\mathcal C}(X)\}$.  Therefore also the complexity of isomorphism of complex structures on $X$ will be $\leq_B$-below the complexity of isomorphism on $SB(X \oplus_\C X)$ (resp. ${\mathcal P}(X \oplus_\C X)$), i.e., between (resp. complemented) subspaces of $X \oplus_\C X$.

This line of ideas initiated with a result of N.J. Kalton proving that if $X \oplus_\C X$ is primary then $X$ admits unique complex structure, and is further exemplified in \cite{Cbis}.

\

Let us note here that since the complexity of linear isomorphism between separable spaces is $E_{\rm max}$, it is natural to ask if such maximum complexity may be achieved by isomorphism between different complex structures on a given real Banach space $X$.

On the other hand,  the relation of isomorphism between complex structures is the orbit relation of the action $(U,T) \mapsto UTU^{-1}$ of the group $GL(X)$ on the standard Borel space ${\cal C}(X)$. So although $GL(X)$ is not a Polish group, one might hope
to prove that isomorphism between complex structures on $X$ is reducible to an orbit relation induced by action of some Polish group.
Our final result is that, similarly to what happens for positions, this is not so:

\begin{thm} There exists a separable real space $X$ such that $E_1$ is Borel reducible to  linear isomorphism between complex structures on $X$. In particular, linear isomorphism between complex structures on  $X$ is not Borel reducible to the orbit equivalence relation induced by a Polish group action on a Polish space.\end{thm}

The proof of the theorem is more technical than the previous ones and is given in the next section. It is striking that the level of complexity $E_1$ may be obtained for  $\C$-linear isomorphism between spaces which are all  $\R$-linearly isometric. This means that in some sense $\R$-linear and $\C$-linear structures may be quite far apart on a Banach space.
Let us note the following question:

\begin{quest} Find a separable real Banach space such that $E_{\rm max}$ is Borel reducible to linear isomorphism between complex structures of $X$.\end{quest}

\section{A reduction result for complex structures}

\noindent Inside $L_p$, with $1\le p <2$, we will construct a (complex) Banach space of the form:
\begin{equation*}
X=\left( \oplus_{k\ge 1}\, X_{k} \right)_{\ell_p} \ {\rm with~each} \  X_k=\left( \oplus_{m\ge 1}\, X_{k,m} \right)_{\ell_p}, \ {\rm for~all} \ k\ge 1.
\end{equation*}
Given $\alpha=(\alpha_k)_k \in (2^{\omega})^{\omega}$, with $\alpha_k=(\alpha_k(m))_m \in 2^{\omega}$, we define for each $k$
$$
X_k(\alpha)=\left( \oplus_{m\ge 1}\, Y_{k,m} \right)_{\ell_p}
$$
where
$$
Y_{k,m}=
\left\{
{\renewcommand{\arraystretch}{1.4}
\begin{array}{l}
X_{k,m}, \ \ {\rm when}\ \ \alpha_k(m)=0\\
{\overline X_{k, m}}, \ \ {\rm when}\ \ \alpha_k(m)=1.
\end{array}}
\right.
$$
Then
$$
X(\alpha)=\left( \oplus_{k\ge 1}\ X_k(\alpha) \right)_{\ell_p}
$$
gives a complex structure on $X$, treated as a real space.

Let $T_\alpha \in {\cal B}(X)$ be the element of the standard Borel space of complex structures on $X$ associated
to $X(\alpha)$. We may write $T_\alpha$ as
$$T_\alpha=(T_k(\alpha))_k,$$
where
$$T_k(\alpha)=(T_{k,m}(\alpha))_m$$
with $T_{k,m}(\alpha)$ is defined on $X_{k,m}$ by
$$
T_{k,m}(\alpha)(x)=
\left\{
{\renewcommand{\arraystretch}{1.4}
\begin{array}{l}
ix, \ \ {\rm when}\ \ \alpha_k(m)=0\\
-ix, \ \ {\rm when}\ \ \alpha_k(m)=1.
\end{array}}
\right.
$$

It is straightforward that the map $\alpha \mapsto T_\alpha$ is Borel from
$(2^{\omega})^{\omega}$ into ${\mathcal B}(X)$ and therefore into the standard Borel space of complex structures on $X$.

\

The claim is that, for a suitable $X$ as above, we have $X(\alpha) \simeq X(\beta)$ (equivalently $T_\alpha \simeq T_\beta$) if and only if $\alpha E_1 \beta$.
For such $X$, $E_1$ is therefore Borel reducible to linear isomorphism between complex structures on $X$. In particular

\begin{thm}
\label{complex_result} The equivalence relation $E_1$ is Borel reducible to  linear isomorphism between complex structures on the subspace $X$ of $L_p$. \end{thm}

\

Let $k\ge 1, m \ge 1$ be arbitrarily fixed. The ingredient space $X_{k,m}$ will be constructed as a subspace of $\ell_{p_1}\oplus_p\ldots \oplus_p \ell_{p_5}$, for some suitable constants
$2>p_1> p_2> \ldots > p_5> p$ depending on $k$ and $m$. Furthermore, $X_{k,m}$ will admit a 1-unconditional decomposition into 2-dimensional spaces $X_{k,m}=\overline {\rm span}~[Z_t]_{t\ge 1}$. More specifically, if we denote by $\{f_{j,t}\}_{t\ge 1}$ the natural basis of $\ell_{p_j}$ ($j=1,
\ldots, 5$), we define the vectors $x_t$ and $y_t$ spanning $Z_t$
by
\begin{equation}
\label{basis}
{\renewcommand{\arraystretch}{1.5}
\begin{array}{lllllll}
x_t & = & f_{1,t} & & +\gamma_1 f_{3,t} & +  \gamma_2 f_{4,t}
& + \ \gamma_3 f_{5,t} \\
y_t & = &  & f_{2,t} & &  + \gamma_2 f_{4,t}
& + i \,\gamma_3 f_{5,t}
\end{array}}
\end{equation}
for all $t\ge 1$. The constants $\gamma_1, \gamma_2, \gamma_3$ will depend on $\eta:= 1/p_2-1/p_1= 1/p_3-1/p_2= 1/p_4-1/p_3=1/p_5-1/p_4$ and on a positive integer $N$ that is chosen according to $k$ and $m$. More precisely, $\gamma_1= N^{-2\eta}, \gamma_2= N^{-8\eta}$ and $\gamma_3= N^{-24 \eta}$.

Under these definitions it was proved in \cite{A} (Corollary 2.2) that, in terms of the Banach-Mazur distance, we have
\begin{equation*}
d(X_{k, m}, {\overline X}_{k,m}) \ge \frac{1}{100} N^{\eta}.
\end{equation*}

This was the consequence of the following fact regarding the behaviour of linear operators acting from $X_{k,m}$ to ${\overline X}_{k,m}$, which will be also useful to us later in the sequence.

First, let $T: W\longrightarrow V$ be a bounded linear operator with $W, V$ Banach spaces having finite dimensional decompositions $\{W_t\}_t$ and $\{V_t\}_t$ respectively. We say that $T$ is {\it block-diagonal} with respect to $\{W_t\}_t$ and $\{V_t\}_t$ if for every $t$ there exists a finite set $B_t \subset \{1, 2, \ldots\}$ such that
$$
\left\{
{\renewcommand{\arraystretch}{1.4}
\begin{array}{l}
{\rm max}\, B_s<{\rm min} B_t\ \ \ \quad \forall s, t \in \{1, 2, \ldots\} ~~ {\rm with}~~ s<t,\\
{\rm supp}\, Tw\subset B_t\ \ \ \quad \quad \forall w\in W_t,\ \forall t\in \{1, 2, \ldots\}
\end{array}}
\right.
$$
where ${\rm supp}\, Tw$ is taken with respect to the decomposition $\{V_t\}_t$.

\begin{pr} (\cite{A}){\ }\
\label{pr-one}
Let $I\subset \{1, 2, \ldots\}$ be an infinite set and let $Y$
 be the subspace of $X_{k, m}$ defined by $Y=\overline {\rm span}
[Z_t]_{t\in I}$.  Consider $T:Y \longrightarrow {\overline
X}_{k, m}$ a block-diagonal operator (with respect to $\{Z_t\}_{t \in I}$
and $\{\overline {Z_t}\}_{t\geq 1}$) with $\|T\|\leq 1$.  Then
\begin{description}
\item[(i)]
There exists a finite set $J\subset I$ such that
\begin{equation}
\label{basis_equation}
{\rm max} \{\|Tx_t\|, \|Ty_t\|\}\leq 24 N^{-\eta},
\quad for\ all \ t\in I\setminus J.
\end{equation}
\item[(ii)] Let $\{I_l\}_{l\geq 1}$ be a family of disjoint subsets of
$I$ with the property that $|I_l|=N$, for all $l\geq1$.
Let $\widetilde {x_l}=\sum_{t\in {I_l}} a_l(t)x_t$,
$\widetilde {y_l}=\sum_{t\in {I_l}} a_l(t)y_t$ satisfy
$\sum_{t\in {I_l}} |a_l(t)|^{p_2}=1$, for $l= 1, 2, \ldots$.  Then
there exists a finite subset $J'\subset \{1,2,...\}$ such that
\begin{equation}
\label{block_estimate}
{\rm max} \{\|T\widetilde{x_l}\|, \|T\widetilde{y_l}\|\}
\leq 70 N^{-\eta},
\quad for\ all \ l\in \{1,2,...\}\setminus J'.
\end{equation}
\end{description}
\end{pr}
Notice that $x_1, y_1, x_2, y_2,\ldots$ form a Schauder basis for both $X_{k,m}$ and ${\overline
X}_{k, m}$. Normalized blocks of this basis satisfy an upper $p_5$-estimate, which in turn implies that the basis is shrinking and hence $w$-null.
Given now any bounded linear operator $T:Y=\overline {\rm span} [Z_t]_{t\in I} \longrightarrow {\overline X}_{k, m}$, with $I\subset \{1, 2, \ldots\}$ infinite, it is easy to see that a classical gliding-hump argument allows us to approximate $T$ on an infinite dimensional subspace $Y_0=\overline {\rm span} [Z_t]_{t\in I_0} \subset Y$ by a block-diagonal operator $T_0: Y_0 \longrightarrow {\overline X}_{k, m}$. We will use this fact later in the sequel.

We are now going into the specific details of choosing the indices $p_1=p_1(k,m)> p_2=p_2(k,m)> \ldots > p_5=p_5(k,m)$ corresponding to each of the ingredient spaces $X_{k, m}$, for $k\ge 1$ and $m\ge 1$, as well as the positive numbers $N=N(k,m)$ which are part of the definition (\ref{basis}) of the basis of $X_{k,m}$. Recall that in (\ref{basis}) we have required
\begin{equation}
\label{p_condition}
\frac{1}{p_2(k,m)}-\frac{1}{p_1(k,m)}= \ldots = \frac{1}{p_5(k,m)}- \frac{1}{p_4(k,m)}=: \eta(k,m).
\end{equation}

We start by picking a sequence $\{q(k,m)\}_{k,m}$ as follows. For a fixed $k\ge 1$ we choose $q(k,m) \longrightarrow p$ as $m\longrightarrow \infty$ such that
$q(k,m)>q(k, m+1)$ for all $m\ge 1$. Once $q(k,m)$ have been chosen for a fixed $k\ge 1$ and for all $m\ge 1$, the inductive step of picking $\{q(k+1,m)\}_{m\ge 1}$ is done in such a way to satisfy
$$
q(k,m+1)> q(k+1,m)> q(1,m+k+1)
$$
for all $m\ge 1$.

It is not hard to see that, once we  choose $\{q(k, m)\}_{k\ge 1, m\ge 1}$ in this way, we can look at it as a decreasing sequence if on the set of double indices $(k,m)$ we consider the order relation "$\le$" which follows the order: $(1,1), \ (1,2), (2,1),$

\noindent $(1,3), (2,2), (3,1), \ \ldots\ $, $\ (1,i), (2, i-1), \ldots, (i-1,2), (i,1),\  \ldots$.

Next, for $k\ge 1$ and $m\ge 1$, we pick $p_1(k,m)> p_2(k,m) > \ldots > p_5(k,m)$ satisfying (\ref{p_condition}) together with
\begin{equation}
\label{main_p_condition1}
q(k,m)> p_1(k,m) >\ldots > p_5(k,m) > q(k_0,m_0)
\end{equation}
where $(k_0,m_0)$ is the successor of $(k,m)$ with respect to "$\le$", and also
\begin{equation}
\label{main_p_condition2}
\frac{(m+k-1)(m+k)}2 \eta(k,m) < \frac 1{p_1(k,m)} -  \frac1{q(k,m)}.
\end{equation}
Lastly, we define $N(k,m)$ as
$$
N(k,m)=\left[(k+1)^{1/\eta(k,m)}\right]
$$
for $k\ge 1$, $m\ge 1$, which gives the estimate
\begin{equation}
\label{N_condition}
k < N(k,m)^{\eta(k,m)} \le  k+1.
\end{equation}


\noindent {\it Proof of Theorem \ref{complex_result}.} \ When $\alpha=(\alpha_k)_k \in (2^{\omega})^{\omega}, \beta=(\beta_k)_k \in (2^{\omega})^{\omega}$ satisfy $\alpha E_1 \beta$ we can see that the spaces $X(\alpha)$ and $X(\beta)$ are isomorphic by means of a linear operator $T: X(\alpha) \longrightarrow X(\beta)$ which is defined as follows:
\begin{itemize}

\item when $\alpha_k(m)=\beta_k(m)$, then $T_{\,|Y_{k,m}}= Id_{\,|Y_{k,m}}$

\item when $\alpha_k(m)=0$ and $\beta_k(m)=1$, then on $X_{k,m}=\overline {\rm span}~[Z_t]_{t\ge 1}$
\begin{equation}
\label{isom1}
T\left(\sum_{t \ge 1} (a_t x_t+ b_t y_t)\right)= \sum_{t \ge 1} (a_t \odot x_t - b_t \odot y_t)
\end{equation}
for all scalars $\{a_t\}_{t \ge 1}$, $\{b_t\}_{t\ge 1}$.

\item when $\alpha_k(m)=1$ and $\beta_k(m)=0$, then on ${\overline X}_{k,m}=\overline {\rm span}~[{\overline Z}_t]_{t\ge 1}$
\begin{equation}
\label{isom2}
T\left(\sum_{t \ge 1} (a_t \odot x_t+ b_t \odot y_t)\right)= \sum_{t \ge 1} (a_t x_t - b_t y_t)
\end{equation}
for all scalars $\{a_t\}_{t \ge 1}$, $\{b_t\}_{t\ge 1}$.

\end{itemize}

Recall that multiplication by scalars in ${\overline X}_{k,m}$ is given by $\lambda \odot x=\overline{\lambda}x$, for $\lambda \in {\bf C}$, $x\in X_{k,m}$. Taking into account the definition (\ref{basis}) of the basis of $X_{k,m}$ and
${\overline X}_{k,m}$ we can rewrite (\ref{isom1}) as
$$
T\left( \sum_{t\ge 1} a_t f_{1,t} + b_t f_{2,t} + \gamma_1 a_t f_{3,t} + \gamma_2 (a_t+ b_t) f_{4,t} + \gamma_3 (a_t+ i b_t) f_{5,t}\right)=
$$
$$
=\sum_{t\ge 1} {\overline a}_t f_{1,t} - {\overline b}_t f_{2,t} + \gamma_1 {\overline a}_t f_{3,t} + \gamma_2 ({\overline a}_t- {\overline b}_t) f_{4,t} +
\gamma_3 ({\overline a}_t - i {\overline b}_t) f_{5,t}.
$$
A simple computation shows that, in this situation, we have
$$
\|T_{\,|X_{k,m}}\| \le 2 \gamma_2 \gamma_3^{-1}= 2 N(k,m)^{\,16 \eta(k.m)}\le 2(k+1)^{16}.
$$
In the case when we are dealing with (\ref{isom2}) we also get $\|T_{\,|{\overline X}_{k,m}}\| \le 2(k+1)^{16}$. Thus we can conclude that
$$
\|T\| \le 2 \left(1+ \max \{k \ |\ \alpha_k \neq \beta_k\}\right)^{16}.
$$

Now let $\alpha=(\alpha_k)_k \in (2^{\omega})^{\omega}$, with $\alpha_k=(\alpha_k(m))_m \in 2^{\omega}$, and $\beta=(\beta_k)_k \in (2^{\omega})^{\omega}$, with $\beta_k=(\beta_k(m))_m \in 2^{\omega}$, be elements in $(2^{\omega})^{\omega}$ which are not $E_1$-equivalent. Without loss of generality assume that
$$
A= \{k \ |\ \exists \, m \ {\rm such~that~} \alpha_k(m)=0 \ {\rm and}\  \beta_k(m)=1\}
$$
is infinite.

Suppose that $T: X(\alpha) \longrightarrow X(\beta)$ is an isomorphism with $\| T\| \le 1/4$ and $\|T^{-1}\|=: C$. For $k\ge 1$, $m\ge 1$, denote by $P_{k,m}: X(\beta) \longrightarrow Y_{k, m}$ the canonical projection of $X(\beta)$ onto
$$
Y_{k,m}=
\left\{
{\renewcommand{\arraystretch}{1.4}
\begin{array}{l}
X_{k,m}, \ \ {\rm when}\ \ \beta_k(m)=0\\
{\overline X_{k, m}}, \ \ {\rm when}\ \ \beta_k(m)=1.
\end{array}}
\right.
$$
Furthermore, since $Y_{k,m} \subset \ell_{p_1(k,m)}\oplus_p\ldots \oplus_p \ell_{p_5(k,m)}$, we will denote by $$Q_j(k,m): Y_{k,m} \longrightarrow \ell_{p_j(k,m)}$$ the canonical projection for all $j=1, \ldots, 5$.

Let $k\in A$ be arbitrarily fixed and pick $m\ge 1$ such that $\alpha_k(m)=0$ and $\beta_k(m)=1$. We will concentrate our attention on the action of the isomorphism $T$ on $X_{k,m}=\overline {\rm span}~[Z_t]_{t\ge 1}$.

First, we notice that for $(k', m') > (k,m)$ we have the following: for every $\delta >0$ and infinite set $L\subseteq {\bf N}$, there exists an infinite subset $L' \subseteq L$ such that
$$
 \|P_{k',m'}T _{\,|\, \overline {\rm span}[Z_t]_{t\in L'}}\| \le \delta.
$$
Otherwise we can find $\delta_0>0$ and a normalized block sequence $(z_s)_s \subset X_{k,m}$ (with respect to the UFDD $\{Z_t\}_t$) satisfying
\begin{equation}
\label{p_1-estimate}
\delta_0< \|P_{k',m'}Tz_s\|~\left(=\left(\|Q_1(k',m') P_{k',m'}Tz_s\|^p+\ldots+\|Q_5(k',m') P_{k',m'}Tz_s\|^p\right)^{\frac1p}\right).
\end{equation}
By passing to a subsequence and perturbing the operator $P_{k',m'}T$ (similarly as in the remark following Proposition \ref{pr-one}) we may assume that $(P_{k',m'}Tz_s)_s$
are successive in $Y_{k',m'}$, with respect to the 2-dimensional UFDD in $Y_{k', m'}$. This  ensures that $(P_{k',m'}Tz_s)_s$ admit a lower $p_1(k',m')$-estimate,
based on (\ref{p_1-estimate}). On the other hand, $(z_s)_s$ admit an upper $p_5(k,m)$-estimate, and this gives a contradiction since $p_5(k,m) >q(k',m') >p_1(k',m')$.

Inductively, for every $(\delta_{k',m'})_{(k',m')> (k,m)} \searrow 0$ we can get infinite sets $\{L_{k',m'}\}_{(k',m') >(k,m)}$ with the property that
$$
L_{k',m'} \supseteq L_{k'',m''},  \qquad {\rm whenever~} (k',m')\le (k'',m''),
$$
and
$$
\|P_{k',m'}T\,_{|\ \overline {\rm span} [Z_t]_{t\in L_{k',m'}}}\| \le \delta_{k',m'},\quad {\rm
for\ all\ } (k',m')>(k,m).
$$
Let $I=\{t_{k',m'}\}_{(k',m') >(k,m)}$ be the diagonal sequence of $\{L_{k',m'}\}_{(k',m') >(k,m)}$. We then obtain a subspace of $X_{k,m}$, namely
$S_{k,m} :=\overline {\rm span}[Z_t]_{t\in I}$, and by a perturbation argument we can get a linear operator
$T_0:S_{k,m} \longrightarrow X(\beta)$ which satisfies
$$
\frac1{2C} \|x\|\leq \|T_0x\|\leq \frac12 \|x\|,\quad {\rm for\ all\ } x\in S_{k,m},
$$
and
$$
P_{k',m'}T_0z=0,
$$
for all $z\in \overline {\rm span} [Z_t]_{t\in I, \,t\ge t_{k',m'}}$ and all $(k',m')>(k,m)$.

Denote by $R_{k,m}: X(\beta)\longrightarrow \left( \sum_{(k',m')>(k,m)} \oplus \, Y_{k',m'}\right)_{\ell_p}$ the cannonical projection.

It is easy to see now that for all $\delta >0$ and every infinite set $L\subseteq I$ there exists $t\in L$ satisfying
$$
\|R_{k,m}{T_0}_{\,|\,Z_t}\|\leq \delta.
$$
Otherwise, we can find $\delta_0 >0$, an infinite set $L_0\subseteq I$ and, for each $t\in L_0$, normalized elements  $z_t\in Z_t$  such that
$\|R_{k,m}T_0z_t\|>\delta_0$. By passing to a subsequence and perturbing the operator $R_{k,m}T_0$ we may assume that
$( R_{k,m}T_0z_t )_{t\in L_0}$ are disjoint in $\left( \sum_{(k',m')>(k,m)} \oplus \, Y_{k',m'}\right)_{\ell_p}$
and thus they admit a lower $p$-estimate. However $(z_t)_{t\in L_0}$ admit an upper $p_5(k,m)$-estimate,
and this gives a contradiction since $p_5(k,m)>p$.

This allows us to obtain a subsequence $\tilde I$ of $I$ and, after some
perturbations, a linear operator (denoted again by) $T_0:\overline {\rm span} [Z_t]_{t\in \tilde
I} \longrightarrow X(\beta)$ with the property that $R_{k,m}T_0=0$ and
$$
\frac1{4C} \|x\|\leq \|T_0x\|\leq \|x\|, \quad \mbox{\rm for\ all\ }
x\in \overline {\rm span} [Z_t]_{t\in \tilde I}.
$$
In addition, we can also assume that $T_0$ has the property that $P_{k',m'}T_0: \overline {\rm span} [Z_t]_{t\in \tilde I} \longrightarrow
Y_{k',m'}$ is block-diagonal, with respect to their respective 2-dimensional decompositions, for all $(k',m')\le (k,m)$.

Looking now at $P_{k,m}T_0:\overline {\rm span}[Z_t]_{t\in \tilde I} \longrightarrow
Y_{k,m}= {\overline X_{k, m}}$ we have all the conditions of Proposition $\ref{pr-one}$ satisfied.
We can then find $I_0 \subset \tilde I$, $|I_0|=N(k,m)$ with the property that, for $y=\sum_{k\in I_0} y_k$,
$$
\|P_{k,m}T_0y\|\le  70 N(k,m)^{-\eta(k,m)} N(k,m)^{\frac1{p_2(k,m)}}.
$$
and thus
$$
\|\sum_{(k',m') < (k,m)} P_{k',m'} T_0y\|\ge \|T_0y\|-\|P_{k,m}T_0y\|
$$
\begin{equation}
\label{operator_condition}
\ge \left(\frac1{4C}- 70N(k,m)^{-\eta(k,m)}\right)
N(k,m)^{\frac1{p_2(k,m)}}.
\end{equation}
For every $(k',m') < (k,m)$ we have
\begin{eqnarray*}
\|P_{k',m'}T_0y\|&=&\|\sum_{k\in I_0} P_{k',m'}T_0y_k\| \\
&\le& \|\sum_{k\in I_0} Q_1(k',m')P_{k',m'}T_0y_k\|+\ldots+\|\sum_{k\in I_0} Q_5(k',m')P_{k',m'}T_0y_k\|\\
&\le & 2 {N(k,m)}^{\frac1{p_1(k',m')}}+\ldots+2 {N(k,m)}^{\frac1{p_5(k',m')}}
\le 10{N(k,m)}^{\frac1{q(k,m)}}.
\end{eqnarray*}
The last inequalities are consequences of (\ref{main_p_condition1}) and the fact that $P_{k',m'}T_0$ is block-diagonal and
$$
\|Q_j(k',m')P_{k',m'}T_0y_k\|\le \|y_k\|\le2,\ \ \forall k\in I_0,\ \ \forall
j=1, \ldots, 5.
$$

Since there are at most $(m+k-1)(m+k)/2$ elements $(k',m')< (k,m)$ we get, as a consequence of (\ref{main_p_condition2}) and (\ref{N_condition}),
$$
\|\sum_{(k',m') < (k,m)} P_{k',m'} T_0y\|\le 10\frac{(m+k-1)(m+k)}2 {N(k,m)}^{\frac1{q(k,m)}}
$$
$$
\le 10k^{\frac{(m+k-1)(m+k)}2} {N(k,m)}^{\frac1{q(k,m)}} <10{N(k,m})^{\frac{(m+k-1)(m+k)}2 \eta(k,m)+\frac1{q(k,m)}}
$$
$$
\le 10 {N(k,m})^{\frac1{p_1(k,m)}}. \hskip8.3cm
$$

We now conclude based on (\ref{operator_condition}) that
$$
10{N(k,m)}^{\frac1{p_1(k,m)}-\frac1{p_2(k,m)}}\ge \frac1{4C}-70{N(k,m)}^{-\eta(k,m)}
$$
which in turn gives $C\ge {N(k,m)}^{\eta(k,m)}/320> k/320$ (by (\ref{p_condition}) and (\ref{N_condition})).

As $k\in A$ was arbitrarily fixed, we obtain a contradiction and this concludes the proof.
\qed

\vskip.3cm

\address


\begin{thebibliography}{99}

\bibitem{A} R. Anisca, \emph{Subspaces of $L_p$ with more than one complex structure},
Proc. Amer. Math. Soc. 131(2003), 2819-2829.



\bibitem{Bo} B. Bossard, {\em A coding of separable Banach spaces. Analytic and coanalytic families of Banach spaces}, Fund. Math. 172 (2002), no. 2, 117--152.


\bibitem{Bou} J. Bourgain, \emph{Real isomorphic complex Banach spaces need not be complex isomorphic}, Proc. Amer. Math. Soc. 96 (1986), 221--226.



\bibitem{CFM} J.M.F.~Castillo, V.~Ferenczi and Y.~Moreno, \emph{On Uniformly Finitely Extensible Banach spaces}, J. Math. Anal. Appl. 410 (2014), 670--686.



\bibitem{CM}
J.M.F.~Castillo and Y.~Moreno, \emph{On the
Lindenstrauss-Rosenthal theorem}, Israel J. Math. 140 (2004),
253--270.




\bibitem{castplich}
J.M.F.~Castillo and  A.~Plichko, \emph{Banach spaces in various positions}, J. Funct. Anal. 259 (2010), 2098--2138.


\bibitem{CC} W. Cuellar Carrera, \emph{A Banach space with a countable infinite number of complex structures}, J. Funct. Anal. 267 (5) (2014), 1462--1487.

\bibitem{Cbis} W. Cuellar Carrera, \emph{Complex structures on Banach spaces with a subsymmetric basis}, J. Math. Anal. and App. 440 (2) (2016), 624-–635.

\bibitem{F} V. Ferenczi, \emph{Uniqueness of complex structure and real hereditarily indecomposable Banach spaces}, Adv. Math. 213 (2007), 462--488.


\bibitem{FLR} V. Ferenczi, A. Louveau, and C. Rosendal, {\em The complexity of classifying Banach spaces up to isomorphism}, Journal of the London Math. Soc 79 (2009), no. 2, 323--345.
\bibitem{FS} H. Friedman and L. Stanley. {\em A Borel reducibility theory for classes of countable structures}. J. Symbolic Logic 54 (1989), no. 3, 894--914.

\bibitem{HKL} L. A. Harrington, A. S. Kechris, and A. Louveau, {\em A Glimm-Effros dicho- tomy for Borel equivalence relations}, J. Amer. Math. Soc. 3 (1990), no. 4, 903--928.



\bibitem{Kechris}
A. S. Kechris, {\em Classical descriptive set theory}, Graduate Texts in Mathe- matics, 156. Springer-Verlag, New York, 1995. xviii+402 pp.

\bibitem{KL}
A. S. Kechris and A. Louveau, {\em The classification of hypersmooth Borel equivalence relations}, J. Amer. Math. Soc. 10 (1997), no. 1, 215--242.



\bibitem{lindtzaf1}
J.~Lindenstrauss and L.~Tzafriri, \emph{Classical Banach spaces
I}, Springer-Verlag, Berlin 1977.

\bibitem{LR} A. Louveau and C. Rosendal, {\em Complete analytic equivalence relations}, Trans. Amer. Math. Soc. 357 (2005), 4839-–4866.

\bibitem{melleray} J. Melleray, {\em Computing the complexity of the relation of isometry between separable Banach spaces}, MLQ Math. Log. Q. 53 (2007), 128-–131.

\bibitem{MP}
Y.~Moreno and A.~Plichko, \emph{On automorphic Banach spaces},
Israel J. Math. 169 (2009), 29--45.

\bibitem{P} A. Pe\l czy\'nski, \emph{Universal bases}, Studia Math. 32 (1969), 247–-268.

\bibitem{silver}J. H. Silver, {\em Counting the number of equivalence classes of Borel and coanalytic equivalence relations}, Ann. Math. Logic 18 (1980), no. 1, 1–-28.
\bibitem{Sz} S. Szarek, \emph{On the existence and uniqueness of complex structure and spaces with "few" operators}, Trans. Amer. Math. Soc. 293 (1986), no. 1, 339-–353.
\bibitem{Zi} J. Zielinski, {\em The complexity of the homeomorphism relation between compact metric spaces}, Advances in Mathematics 291 (2016), 635--645.
\end{thebibliography}
\end{document}